\documentclass[a4paper]{amsart}
\usepackage{amsgen, amstext,amsbsy,amsopn, amsthm, amsfonts,amssymb,amscd,amsmath,euscript,enumerate,url,verbatim,calc,xypic,tikz}
\usepackage{enumitem}
 \usepackage{xcolor}
 \usepackage{tikz-cd}
\usepackage{amsgen, amstext, amsbsy, amsopn, amsthm, amsfonts, amssymb, amscd, amsmath, euscript, enumerate, url, verbatim, calc, xypic, tikz}
\usepackage{amsfonts}
\usepackage{amssymb} 
\usepackage{graphicx}

\usepackage{mathrsfs}  
\usepackage{hyperref}
\usepackage{MnSymbol}
\usepackage{pgfkeys}
\usepackage{mathtools}
\usepackage{hyperref}
\usepackage{MnSymbol}
\usepackage{pgfkeys}
\usepackage{mathtools}
\def\multiset#1#2{\ensuremath{\left(\kern-.3em\left(\genfrac{}{}{0pt}{}{#1}{#2}\right)\kern-.3em\right)}}

\usetikzlibrary{arrows}

\usepackage{latexsym}
\usepackage{graphics}
\usepackage{color}
\usepackage{lastpage}
\usepackage{fancyhdr}
\usepackage{multirow}
\allowdisplaybreaks
\usepackage{graphicx}
\graphicspath{ {F:/IMAGES/} }

\newcommand{\m}{\mathfrak{m} }

\newcommand{\ov}{\overline}

\newcommand{\type}{\operatorname{type}}

\newcommand{\depth}{\operatorname{depth}}

\newcommand{\tpe}{\operatorname{type}}

\newcommand{\Ext}{\operatorname{Ext}}

\newcommand{\lm}{\lambda}

\newcommand{\proset}{\,\mathrel{\lower 4pt\hbox{$\scriptscriptstyle/$}
		\mkern -14mu\subseteq }\,} 
		\usepackage{stackengine}
\newcommand\tsup[2][2]{%
 \def\useanchorwidth{T}%
  \ifnum#1>1%
    \stackon[-.5pt]{\tsup[\numexpr#1-1\relax]{#2}}{\scriptscriptstyle\sim}%
  \else%
    \stackon[.5pt]{#2}{\scriptscriptstyle\sim}%
  \fi%
}
\newtheorem{theorem}{Theorem}[section]
\newtheorem{corollary}[theorem]{Corollary}
\newtheorem{lemma}[theorem]{Lemma}
\newtheorem{proposition}[theorem]{Proposition}

\newtheorem{question}[theorem]{Question}

\usepackage{amsmath}

\theoremstyle{definition}

\newtheorem{remark}[theorem]{Remark}

\newtheorem{example}[theorem]{Example}
\title[Bounds for reduction number] {Bounds for the reduction number of primary ideal in dimension three}
\author[M Mandal and K Saloni]{Mousumi Mandal and Kumari Saloni}

\address{Department of Mathematics, Indian Institute of Technology Kharagpur, 721302, India} \email{mousumi@maths.iitkgp.ac.in}
\address{Department of Mathematics, Indian Institute of Technology Patna, Bihta, Patna 801106, India}
\email{ksaloni@iitp.ac.in}
\begin{document}
\maketitle
\begin{abstract}
Let $(R,\m)$ be a Cohen-Macaulay local ring of dimension $d\geq 3$ and $I$ an $\m$-primary ideal of $R$. Let $r_J(I)$ be the reduction number of $I$ with respect to a minimal reduction $J$ of $I$. 
Suppose $\depth G(I)\geq d-3$. We prove that $r_J(I)\leq e_1(I)-e_0(I)+\lm(R/I)+1+(e_2(I)-1)e_2(I)-e_3(I)$, where $e_i(I)$ are  Hilbert coefficients. Suppose $d=3$ and $\depth G(I^t)>0$ for some $t\geq 1$. Then we prove that  $r_J(I)\leq e_1(I)-e_0(I)+\lm(R/I)+t$. 
%
%
\end{abstract}
%
%

%
%
%
\section{Introduction}
Let $(R,\m)$ be a Noetherian local ring of dimension $d\geq 1$ and $I$ an $\m$-primary ideal. 
 A sequence of ideals $\mathcal I=\{I_n\}_{n\in \mathbb Z}$ is called an {\it $I$-admissible  filtration} \index{ $I$-admissible  filtration} if $(i)$ $I_{n+1}\subseteq I_n$,
   $(ii)$  $I_m I_n\subseteq I_{m+n}$ and $(iii)~ I^n\subseteq I_n\subseteq I^{n-k}$ for some $k\in \mathbb N$. A {\it reduction} \index{reduction} of $\mathcal I=\{I_n\}$ is an ideal $J\subseteq I_1$ such that $JI_n=I_{n+1}$ for  $n\gg 0$ and it is called {\it minimal reduction} if it is  minimal with respect to containment among all reductions. A Minimal reduction of $\mathcal I$ exists and is generated by $d$ elements if $R/\m$ is infinite. Minimal reductions are important in the study of Hilbert functions and blow-up algebras. For a minimal reduction $J$ of $\mathcal I$, we define \[r_J(\mathcal I)= \sup \{n\in \mathbb Z \mid I_n\not= JI_{n-1}\} \mbox{ and } 
  r(\mathcal I)=\min\{r_J(\mathcal I)~|~ J \mbox{ is a minimal reduction of } \mathcal I\},\]
known as {\it{reduction number of $\mathcal I$ with respect to $J$}} and  {\it{reduction number of $\mathcal I$}}  respectively. 
We write $r_J(I)$ and $r(I)$ in place of $r_J(\mathcal I)$ and $r(\mathcal I)$ respectively for $\mathcal I=\{I^n\}_{n\in\mathbb{Z}}$. 
Reduction number is an important data associated to an ideal which contains information about the depth and structural properties of the associated graded ring $G(\mathcal I)=\mathop\oplus\limits_{n\geq 0}I_n/I_{n+1}$. The number $r_J(I)$ can be seen as a measure of how closely $J$ and $I$ are related. In general, it may be hard to compute $r_J(I)$. We look for bounding $r_J(I)$ in terms of other invariants of the ring or the ideal such as embedding dimension, Hilbert coefficients etc. 
The Hilbert coefficients of $\mathcal I$ are the unique integers $e_i(\mathcal I)$, $0\leq i\leq d,$ such that the function  $H_{\mathcal I}(n):=\lm(R/I_n)$ coincides with the following polynomial for $n\gg 0:$  $$P_{\mathcal I}(x) =e_0(\mathcal I){x+d-1\choose d}-e_1(\mathcal{ I}){x+d-2\choose d-1}+\cdots +(-1)^de_d(\mathcal I).$$ 
Here $\lm(*)$ denotes the length function. The function $H_{\mathcal I}(n)$ and the polynomial $P_{\mathcal I}(x)$ are known as the Hilbert-Samuel function and the Hilbert-Samuel polynomial of $\mathcal I$ respectively.  For $\mathcal I=\{I^n\}$, we write $e_i(I)$ instead of $e_i(\mathcal I)$. We refer to \cite{rv} for the related background material. \\

It is well known that if $\depth G(I) \geq d-1$, then $r_J ( I)$ does not depend on the minimal reduction $J$. Further, if $R$ is a one dimensional Cohen-Macaulay local ring then $r(I)\leq e_0(I)-1$. 
In \cite[Theorem 2.45]{v}, Vasconcelos proved that in a Cohen-Macaulay local ring of dimension $d\geq 1$, \begin{equation}\label{theorem-intro-old1}
    r(I)\leq \frac{d. e_0(I)}{o(I)}-2d+1
\end{equation}  where  $o(I)$  is the largest positive integer $n$ such that $I \subseteq \m^n$. 

A non-Cohen-Macaulay version of the above result can be found in \cite[Theorem 3.3]{gghv13}. Let $R$ be a Cohen-Macaulay local ring of dimension at most two. In \cite[Corollary 1.5]{r}, Rossi proved that 
\begin{equation}\label{rossi-bound} 
r_J(I)\leq e_1(I)-e_0(I)+\lm(R/I)+1
\end{equation} for a minimal reduction $J\subseteq I.$
Since then many attempts have been made to achieve a bound of similar character in higher dimension. For instance,  the  bound in \eqref{rossi-bound} holds in all dimensions if $\depth G(I)\geq d-2$ \cite[Theorem 4.3]{rv} or if $e_1(I)-e_0(I)+\lm(R/I)=1$ \cite[Theorem 3.1]{gno08}. Another case is when $I\subseteq k[x,y,z]$ is of codimension $3$ generated by five quadrics \cite[Theorem 2.1 and Proposition 2.4]{hsv15}. However, no example is known to counter the 
relation in \eqref{rossi-bound}
in higher dimension. In this paper, our objective is to find bounds for $r_J(I)$ in dimension three involving higher Hilbert coefficients.  We prove the following generalization of Rossi's result \cite[Theorem 4.3]{rv} in dimension three: 
\begin{theorem}\label{theorem-intro-1}
Let $(R,\m)$ be a Cohen-Macaulay local ring of dimension three and $I$ an $\m$-primary ideal with $\depth G(I^t)>0$ for some $t\geq 1.$ Let $J\subseteq I$ be a minimal reduction of $I$.  Then
 \begin{equation}
 r_J(I)\leq e_1(I)-e_0(I)+\lm(R/I)+t.\end{equation}
 Furthermore, if $r_J(I)\equiv k \mod t$, $1\leq k\leq t-1$, then 
  \begin{equation*} r_J(I)\leq e_1(I)-e_0(I)+\lm(R/I)+k.\end{equation*}
\end{theorem}
As a consequence, if $r_J(I)$ is odd and $\depth G(I^2)>0$ then 
 $r_J(I)\leq e_1(I)-e_0(I)+\lm(R/I)+1.$ Furthermore, we prove the following bound in dimension $d\geq 3$.  
\begin{theorem}\label{theorem-intro-2}
Let $(R,\m)$ be a Cohen-Macaulay local ring of dimension $d\geq 3$  and $I$ an $\m$- primary ideal  with $\depth G(I)\geq d-3$. Then 
\begin{equation}\label{our-main-bound}
    r_J(I)\leq e_1(I)-e_0(I)+\lm(R/I)+1+(e_2(I)-1)e_2(I)-e_3(I).
\end{equation}
\end{theorem}
Though the bound in \eqref{our-main-bound} is quadratic in $e_2(I)$, we illustrate various examples where it is tighter than earlier known bounds. For small values of $e_2(I)$ in dimension three, we obtain linear bounds for $r_J(I)$ in terms of $e_i(I)$, $0\leq i\leq 3$, in Corollary \ref{corr-main-result-1}. 
\\

It is worth asking if we can find similar bounds for $r_J(I)$ in Noetherian local rings. In \cite{gghv}, Ghezzi et. al. proved a nonlinear bound in terms of Hilbert coefficients  in two dimensional Buchsbaum local rings of positive depth. We prove the following results:
\begin{theorem}\label{theorem-intro-3}
Let $(R,\m)$ be a Buchsbaum local ring of dimension $d\leq 2$ and $I$ an $\m$-primary ideal. \begin{enumerate}
    \item Let $d=1$. Then $r_J(I)\leq e_1(I)-e_1(J)-e_0(I)+\lm(R/I)+2$.
    \item Let $d=2$ and $\depth G(I^t)>0$ for some $t\geq 1$. Then  $r_J(I)\leq e_1(I)-e_1(J)-e_0(I)+\lm(R/I)+t+1$.
\end{enumerate}
\end{theorem}

The main difficulty in generalizing the bound in \eqref{rossi-bound} in higher dimension is that  $r_J(I)$ does not behave well with respect to superficial elements. This fact is closely related to Ratliff-Rush closure $\{\widetilde{I^n}\}$ of powers of $I$, see Lemma \ref{l1}. We recall the definition more generally for an $I$-admissible filtration $\mathcal I$. The Ratliff-Rush filtration  of $\mathcal I$ is the filtration $\{\widetilde{I_n}=\mathop\bigcup\limits_{t\geq 0}(I_{n+t}:I^t)\}_{n\in\mathbb{Z}}$. 
For  a minimal reduction $J$ of $\mathcal I$, we set $$\widetilde{r}_J(\mathcal I):=\sup\{n\in \mathbb Z|\widetilde{I_{n}}\not=J\widetilde{I_{n-1}}\}.$$ We write $\widetilde{r}_J(I)$ if $\mathcal{I}=\{I^n\}$. Note that if $\depth G(I)>0$, then $\widetilde{r}_J(I)=r_J(I)$.  In \cite{rs}, Rossi and Swanson proved that in a Cohen-Macaulay local ring of dimension two $\widetilde{r}_J(I)\leq r_J(I)$. It follows that $\widetilde{r}_J(I)\leq e_1(I)-e_0(I)+\lm(R/I)+1$ in dimension two. We extend the result of Rossi and Swanson for any $I$-admissible filtration in Proposition \ref{mafi}.

It is natural to ask if $r_J(\mathcal I)\leq  e_1(\mathcal I)-e_0(\mathcal 
I)+\lm(R/I_1)+1$ for an $I$-admissible filtration $\mathcal I$. This is largely unknown, even in smaller dimensions. 
In Theorem \ref{lemma-bound-on-rtilde}, we prove that $\widetilde{r}_J(\mathcal I)\leq e_2(\mathcal I)+1$ and subsequently discuss the cases when  $\widetilde{r}_J(\mathcal I)= e_2(\mathcal I)+1$ and $\widetilde{r}_J(\mathcal I)=e_2(\mathcal I)$.
\\


This paper is organised in four sections. 
In Section \ref{section-2}, we prove  $\widetilde{r}_J(\mathcal I)\leq r_J(\mathcal I)$ 
and discuss bounds on $\widetilde{r}_J(\mathcal I)$ in two dimensional Cohen-Macaulay local rings. In Section \ref{section-3}, we establish Theorem \ref{theorem-intro-1} and its consequences. Then we  gather various cases when the bound $r_J(I)\leq e_1(I)-e_0(I)+\lm(R/I)+1$ holds in dimension three. 
We also prove Theorem \ref{theorem-intro-3} in this section.
In Section \ref{section-4}, we prove Theorem \ref{theorem-intro-2}.
Examples are given in support of our bounds being better than the earlier known bounds.

\section{Bound for $\widetilde{r}_J(\mathcal I)$ in dimension two}\label{section-2}
For this section, let $R$ be a Cohen-Macaulay local ring of dimension two. Inspired by Rossi's bound in \eqref{rossi-bound},  one can ask whether $r_J(\mathcal I)\leq e_1(\mathcal{I})-e_0(\mathcal I)+\lm(R/I_1)+1$ for an $I$-admissible filtration $\mathcal I$? In general, it is not known. Suppose $I=\tilde{I}$. Then for the case $\mathcal{I}=\{\widetilde{I^n}\}_{n\in\mathbb{Z}}$, the above question has affirmative answer  which follows from a result of Rossi and Swanson \cite[Proposition 4.5]{rs}. They proved that $\widetilde{r}_J(I)\leq r_J(I)$. We generalize \cite[Proposition 4.5]{rs} for an $I$ admissible filtration. Further in this section, we prove certain bounds for $\widetilde{r}_J(\mathcal I)$. 
\begin{proposition}\label{mafi}
Let $(R,\m)$ be a Cohen-Macaulay local ring of dimension two, $I$ an $\m$-primary ideal and $\mathcal I=\{I_n\}$ an $I$-admissible filtration. Then, for a minimal reduction $J$ of $\mathcal I$,  $\widetilde{r}_J(\mathcal I)\leq r_J(\mathcal I)$.
\end{proposition}
\begin{proof}
Let $r=r_J(\mathcal I)$ and $J=(x,y)$ with $x,y$  a regular sequence in $R$. We show that for all $n\geq r$, $\widetilde{I_{n+1}}=J\widetilde{I_{n}}.$ For $k\gg 0$, we may write $\widetilde{I_{m}}=(I_{m+k}:(x^k,y^k))$ for $m=n-1,n,n+1$. Let $a\in \widetilde{I_{n+1}}= (I_{n+1+k}:(x^k,y^k))=(J^{k+1}I_n:(x^k,y^k))$. Then $ax^k\in J^{k+1}I_n\subseteq x^k xI_n+yI_{n+k}.$ Let $ax^k=x^kb+yc$ where $b\in xI_n$ and $c\in I_{n+k}.$ This gives $x^k(a-b)=yc$. Since $x,y$ is a regular sequence we have $a-b=dy$ and $c=x^kd$ for some $d\in R$. As $c\in I_{n+k}$, we get $d\in (I_{n+k}:x^k)$ and $a-b=dy\in y(I_{n+k}:x^k)$. Therefore, $$a\in xI_n+y(I_{n+k}:x^k).$$ 
By similar arguments, we can show that $a\in yI_n+x(I_{n+k}:y^k)$. Now let $a=xr_1+ys_1=yr_2+xs_2$ where $r_1, r_2\in I_n$,  $s_1\in (I_{n+k}:x^k)$ and $s_2\in  (I_{n+k}:y^k)$. Then $x(r_1-s_2)=y(r_2-s_1)$ which implies $r_2-s_1=\alpha x$ and $r_1-s_2=\alpha y$ for some  $\alpha\in R$. Then $\alpha x^{k+1}=r_2 x^k-s_1x^k\in I_{n+k}$ and $\alpha y^{k+1}=r_1y^k-s_2 y^k\in I_{n+k}$. These give $\alpha\in  (I_{n+k}:(x^{k+1},y^{k+1}))=\widetilde{I_{n-1}}.$ Therefore $r_1-s_2=\alpha y\in J\widetilde{I_{n-1}}\subseteq \widetilde{I_{n}}$ and $r_2-s_1=\alpha x\in J\widetilde{I_{n-1}}\subseteq \widetilde{I_{n}}$. This gives $s_1,s_2\in  \widetilde{I_n}$ and $a=xr_1+ys_1\in JI_n+J\widetilde{I_n}\subseteq J\widetilde{I_n}$. Therefore $\widetilde{I_{n+1}}=J\widetilde{I_{n}}$ for all $n\geq r$. 
\end{proof}
In Example \ref{example-n}, we see a class of examples with $\widetilde{r}_J(\mathcal I)< r_J(\mathcal I)$.  In the next theorem, we give an upper bound for $\widetilde{r}_J({\mathcal I})$ in dimension two. We may assume from now on that $I\not=J$. 
\begin{theorem}\label{lemma-bound-on-rtilde}
Let $(R,\m)$ be a two dimensional Cohen-Macaulay local ring, $I$ an $\m$-primary ideal and $\mathcal{I}=\{I_n\}$ an $I$-admissible filtration. Then, for a minimal reduction $J\subseteq I$,  $$\widetilde r_J({\mathcal{I}})\leq e_2(\mathcal{I})+1.$$
 Furthermore, consider the following statements:
\begin{enumerate}[label=(\roman*)]
    \item \label{roman-1} $\widetilde r_J({\mathcal{I}})=e_2(\mathcal{I})+1$;
    \item \label{roman-2} $\widetilde{I_{n+1}}=J\widetilde{I_n}$ for all $n\neq0,e_2(\mathcal{I})$;
    \item \label{roman-3}$\lm(\widetilde{I_{n+1}}/J\widetilde{I_n})=1$ for $n=e_2(\mathcal{I})$;
    \item \label{roman-4} $e_1(\mathcal{I})=e_0(\mathcal{I})-\lm(R/\widetilde{I_1})+1$.
\end{enumerate}
We have \ref{roman-4} $\implies$ \ref{roman-3} $\implies$ \ref{roman-1} $\Longleftrightarrow$ \ref{roman-2} and all four are equivalent if $e_2(\mathcal{I})\neq 0.$ 
\end{theorem}
\begin{proof}
Since $\depth G(\widetilde{\mathcal{I}})\geq 1,$ we have  $e_2(\mathcal{I})=\mathop\sum\limits_{n\geq 0}nv_n(\widetilde{\mathcal{I}})$ by \cite[Theorem 2.5]{rv}, where $v_n(\widetilde{\mathcal I})=\lm(\widetilde{I_{n+1}}/J\widetilde{I_n})$. This gives $nv_n(\widetilde{\mathcal{I}})=0$ for all $n\geq e_2(\mathcal{I})+1$. Hence $\widetilde{{I}_{n+1}}=J\widetilde{{I}_{n}}$  for all $n\geq e_2(\mathcal I)+1$, i.e., $\widetilde r_J({\mathcal{I}})\leq e_2(\mathcal{I})+1$.

Now we show \ref{roman-1}$\Longleftrightarrow$ \ref{roman-2}. Suppose $\widetilde r_J({\mathcal{I}})=e_2(\mathcal{I})+1$. Then $\widetilde{I_{n+1}}=J\widetilde{I_n}$ for all $n\geq e_2(\mathcal{I})+1$ and $\widetilde{I_{n+1}}\neq J\widetilde{I_n}$ for $n=e_2(\mathcal{I})$. This gives \ref{roman-2} when $e_2(\mathcal{I})=0$. When $e_2(\mathcal{I})\neq 0$, we have 
$0< e_2(\mathcal{I})=\sum_{n=0}^{e_2(\mathcal{I})}nv_n(\widetilde{\mathcal{I}})$ with $v_{e_2(\mathcal{I})}(\widetilde{\mathcal{I}})\neq0$ which implies $\lm(\widetilde{I_{n+1}}/J\widetilde{I_n})=1$ for $n=e_2(\mathcal{I})$ and $\widetilde{I_{n+1}}=J\widetilde{I_n}$ for all $1\leq n\leq e_2(\mathcal{I})-1.$ This gives \ref{roman-2}.  For the converse, suppose the assumption in \ref{roman-2} holds. When $e_2(\mathcal{I})=0$, $\widetilde r_J({\mathcal{I}})\leq 1$. In fact, $\widetilde r_J({\mathcal{I}})= 1$. Otherwise, $\widetilde r_J({\mathcal{I}})=0$ gives $\widetilde{I_1}=J$, which is not true. Now suppose $e_2(\mathcal{I})\neq 0$. Then $e_2(\mathcal{I})=e_2(\mathcal{I})\lm(\widetilde{I_{e_2(\mathcal{I})+1}}/J\widetilde{I_{e_2(\mathcal{I})}})$ which implies $\lm(\widetilde{I_{e_2(\mathcal{I})+1}}/J\widetilde{I_{e_2(\mathcal{I})}})=1.$ Therefore $\tilde{r}_J({\mathcal{I}})\geq e_2(\mathcal{I})+1$. Note that the above arguments also prove \ref{roman-1}$\implies$ \ref{roman-3} and \ref{roman-2}$\implies$ \ref{roman-3} when $e_2(\mathcal{I})\neq 0$.

Suppose \ref{roman-3} holds. Then $\widetilde r_J({\mathcal{I}})\geq e_2(\mathcal{I})+1$. Since  $\widetilde r_J({\mathcal{I}})\leq e_2(\mathcal{I})+1$, we get the equality as in \ref{roman-1}.

Suppose \ref{roman-4} holds. Since  $e_1(\mathcal{I})=\sum_{n=0}^{e_2(\mathcal{I})}\lm(\widetilde{I_{n+1}}/J\widetilde{I_n}),$ we have $e_1(\mathcal I)=e_0(\mathcal{I})-\lm(R/\widetilde{I_1})+1$ if and only if $\sum_{n=0}^{e_2(\mathcal{I})}\lm(\widetilde{I_{n+1}}/J\widetilde{I_n})=\lm(\widetilde{I_1}/J)+1$. This forces $e_2(\mathcal{I})\neq 0$ and 
$\sum_{n= 1}^{e_2(\mathcal{I})}\lm(\widetilde{I_{n+1}}/J\widetilde{I_n})=1$. Therefore $\widetilde{I_{n+1}}=J\widetilde{I_n}$ for all $n\geq1$ except one, say $n=n_0$, $1\leq n_0\leq e_2(\mathcal{I})$ and $\lm(\widetilde{I_{n_0+1}}/J\widetilde{I_{n_0}})=1$. Then $e_2(\mathcal{I})=\mathop\sum\limits_{n=0}^{e_2(\mathcal{I})}n\lm(\widetilde{I_{n+1}}/J\widetilde{I_n})=n_0\lm(\widetilde{I_{n_0+1}}/J\widetilde{I_{n_0}})=n_0$. This proves \ref{roman-4}$\implies$ \ref{roman-2} and \ref{roman-4}$\implies$ \ref{roman-3}.

Finally assume $e_2(\mathcal{I})\neq 0$ and \ref{roman-2} holds. Then  we get $e_1(\mathcal{I})=\sum_{n=0}^{e_2(\mathcal{I})}\lm(\widetilde{I_{n+1}}/J\widetilde{I_n})=\lm(\widetilde{I_{1}}/J)+\lm(\widetilde{I_{e_2(\mathcal{I})+1}}/J\widetilde{I_{e_2(\mathcal{I})}})=e_0(\mathcal{I})-\lm(R/\widetilde{I_1})+1$. 
\end{proof}
\begin{corollary}
Let $(R,\m)$ be a two dimensional Cohen-Macaulay local ring and $I$ an $\m$-primary ideal. If $\widetilde r_J( { I})=e_2(I)+1\neq 1$ then $$1\leq e_2(I)\leq r_J(I)-1\leq  \lm(\widetilde I/I)+1.$$ Moreover, if $I$ is Ratliff-Rush closed then the following statements hold:
\begin{enumerate}[label=(\roman*)]
    \item \label{cor-roman-1} $r_J(I)=2$
    \item \label{cor-roman-2} $e_2(I)=1$
    \item \label{cor-roman-3} $\depth G(I)\geq 1$
\end{enumerate}
\end{corollary}
\begin{proof}
 We have 
 \begin{eqnarray}\label{eq1}
    e_2(I)+1&=& \widetilde r_J(I)\nonumber\\
    &\leq & r_J(I) \hspace{1cm}(\mbox{by Proposition \ref{mafi}})\nonumber\\
    &\leq & e_1(I)-e_0(I)+\lm(R/I)+1 \mbox{ (by }\eqref{rossi-bound} )
 \end{eqnarray}
  By Theorem \ref{lemma-bound-on-rtilde}, $e_1(I)=e_0(I)-\lm(R/\widetilde I)+1$. Substituting the value in equation \eqref{eq1}, we get $e_2(I)+1\leq r_J(I)\leq \lm(\widetilde I/I)+2$ which implies 
 \begin{equation*}
    1\leq e_2(I)\leq r_J(I)-1 \leq \lm(\widetilde I/I)+1.
 \end{equation*} 
  Moreover if $I$ is Ratliff-Rush closed, then we obtain $e_2(I)=1$ and $r_J(I)=e_2(I)+1=2$. Then by \cite[Theorem 3.3]{rv} we have $\depth G(I)\geq 1$.
\end{proof}
\begin{corollary}\label{cor-n}
Let $(R,\m)$ be a two dimensional Cohen-Macaulay local ring, $I$ an $\m$-primary ideal and $\mathcal{I}=\{I_n\}$ an $I$-admissible filtration.
For a minimal reduction $J\subseteq I$, 
if $\widetilde r_J({\mathcal{I}})=e_2(\mathcal{I})$
then the following statements hold:
\begin{enumerate}[label=(\roman*)]
    \item \label{cor2-roman-1} $\widetilde{I_{n+1}}=J\widetilde{I_n}$ for  $n\neq0,1,e_2(\mathcal{I})-1$;
    \item \label{cor2-roman-2} $\lm(\widetilde{I_{n+1}}/J\widetilde{I_n})=\begin{cases} 
    2    \text{ for } n=1 \text{ if } e_2(\mathcal{I})=2,\\
    1    \text{ for } n=1,e_2(\mathcal{I})-1 \text{ if } e_2(\mathcal{I})\neq 2,
 \end{cases}$
    \item \label{cor2-roman-3}  $e_1(\mathcal{I})=e_0(\mathcal{I})-\lm(R/\widetilde{I_1})+2$.
\end{enumerate}
\end{corollary}
\begin{proof}
 Note that $\widetilde r_J({\mathcal{I}})=e_2(\mathcal{I})$ if and only if $\widetilde{I_{n+1}}=J\widetilde{I_n}$ for all $n\geq e_2(\mathcal{I})$ and $\widetilde{I_{n+1}}\neq J\widetilde{I_n}$ for $n=e_2(\mathcal{I})-1$. Since $\depth G(\widetilde{\mathcal{I}})\geq 1,$ we have  $e_2(\mathcal{I})=\sum_{n=0 }^{e_2(\mathcal{I})-1}nv_n(\widetilde{\mathcal{I}})$. Therefore $e_2(\mathcal{I})\neq 1$.
 Now suppose $e_2(\mathcal{I})=2.$ Then $\lm(\widetilde{I_{2}}/J\widetilde{I_1})=2$ and $\widetilde{I_{n+1}}=J\widetilde{I_n}$ for $n\geq 2$. For the case $e_2(\mathcal{I})\geq 3$, we get
 $$2(e_2(\mathcal{I})-1)>e_2(\mathcal{I})\geq (e_2(\mathcal{I})-1)v_{e_2(\mathcal{I})-1}(\widetilde{\mathcal{I}}).$$
 So, $v_{e_2(\mathcal{I})-1}(\widetilde{\mathcal{I}})=1$. This gives $1=e_2(\mathcal{I})-(e_2(\mathcal{I})-1)v_{e_2(\mathcal{I})-1}(\widetilde{\mathcal{I}})= \sum_{n=0 }^{e_2(\mathcal{I})-2}nv_n(\widetilde{\mathcal{I}})$ which implies $\lm(\widetilde{I_{n+1}}/J\widetilde{I_n})=1$ for $n=1$ and $\widetilde{I_{n+1}}=J\widetilde{I_n}$ for all $2\leq n\leq e_2(\mathcal{I})-2.$ This proves \ref{cor2-roman-1}  and \ref{cor2-roman-2}.  
To see \ref{cor2-roman-3} , we have $e_1(\mathcal{I})=\sum_{n=0}^{e_2(\mathcal{I})-1}\lm(\widetilde{I_{n+1}}/J\widetilde{I_n})=
e_0(\mathcal{I})-\lm(R/\widetilde{I_1})+2.$ 
\end{proof}
\begin{example}\label{example-n}\cite[Theorem 5.2]{or}
Let $m\geq 0$, $d\geq 2$ and $k$ be an infinite field. Consider the power series ring $D=k[[\{X_j\}_{1\leq j\leq m}, Y, \{V_j\}_{1\leq j\leq d}, \{Z_j\}_{1\leq j\leq d}]]$ with $m+2d+1$ indeterminates and the ideal $\mathfrak{a}=[(X_j~|~1\leq j\leq m)+(Y)].[(X_j~|~1\leq j\leq m)+(Y)+(V_i~|~1\leq i\leq d)]+(V_iV_j~|~1\leq i,j\leq d, i\neq j)+(V_i^3-Z_iY~|~1\leq i\leq d).$ Define $R=D/\mathfrak{a}$ and $x_i,y,v_i,z_i$ denote the images of $X_i, Y, V_i, Z_i$ in $R$ respectively. Let $\mathfrak{m}=(x_j~|~1\leq j\leq m)+(y)+(v_j~|~1\leq j\leq d)+(z_j~|~1\leq j\leq d)$ be the maximal ideal in $R$ and $Q=(z_j~|~1\leq j\leq d).$ Then \begin{enumerate}
    \item $R$ is Cohen-Macaulay local ring with $\dim R=d,$
    \item $Q$ is a minimal reduction of $\m$ with $r_Q(\m)=3,$
    \item $e_0(\m)=m+2d+2;$ $e_1(\m)=m+3d+2$; $e_2(\m)=d+1$ and $e_i(\m)=0$ for $3\leq i\leq d,$ 
    \item $G(\m)$ is Buchsbaum ring with $\depth G(\m)=0$.
\end{enumerate} 
Particularly when $d=2$, we have $e_1(\m)=m+8$ and $e_0(\m)-\lm(R/\widetilde{\m})+2=m+7$ which implies $e_1(\m)\not= e_0(\m)-\lm(R/\widetilde{\m})+2$. Therefore by Corollary \ref{cor-n},  $\widetilde{r}_Q(\m)\neq e_2(\m)=3=r_Q(\m)$. Therefore $\widetilde{r}_Q(\m)< r_Q(\m).$
\end{example}

We end this section with the following questions.
\begin{question}
Is $r_J(\mathcal I)\leq e_1(\mathcal I)-e_0(\mathcal I)+\lm(R/I_1)+1$ for any $I$-admissible filtration in two dimensional Cohen-Macaulay local ring?
Since $\widetilde r_J( {\mathcal I})\leq r_J(\mathcal{I})$ by Proposition \ref{mafi}, one may investigate whether the same bound holds for $\widetilde r_J( {\mathcal I})$? 
\end{question}
\begin{question}
Is $\widetilde r_J( {\mathcal I})\leq r_J(\mathcal{I})$ for $d\geq 3$? 
\end{question}
\section{Rossi's bound  in dimension three}\label{section-3}
In general, reduction number does not behave well with respect to a superficial element $x\in I$, i.e., $r_{J/(x)}(I/(x))= r_J(I)$ may not hold. 
When $\depth G(I)\geq 1$, then  $r_{J/(x)}(I/(x))= r_J(I)$, see \cite[Lemma 2.14]{marley}.
However, there are examples when $r_{J/(x)}(I/(x))= r_J(I)$ and $\depth G(I)=0.$
Note that $\depth G(I)\geq 1$ is equivalent to the condition that $\widetilde{I^n}=I^n$  for all  $n\geq 1$.  In the lemma below, we state a necessary condition for $r_{J/(x)}(I/(x))<r_J(I)$.
\begin{lemma}\label{l1}
Let $(R,\m)$ be a Noetherian local ring of dimension $d\geq 1$ and $\depth R>0$. Let $I$ be  an $\m$ primary ideal and $J\subseteq I$  a minimal reduction of $I$. If $r_{J/(x)}(I/(x))< r_J(I)$ for a superficial element $x\in I$, then $\widetilde{I^n}\neq I^n$ for all $r_{J/(x)}(I/(x))\leq n < r_J(I)$.
\end{lemma}
\begin{proof} 
Suppose $\widetilde{I^n}=I^n$ for some $n$ with  $r_{J/(x)}(I/(x))\leq n < r_J(I)$. Then  $I^n\subseteq (I^{n+1}:x)\subseteq (\widetilde{I^{n+1}}:x)=\widetilde{I^n}=I^n$. Thus $(I^{n+1}:x)=I^n$ which implies $I^{n+1}\cap (x)=xI^n$. On the other hand,  $(I/(x))^{n+1}=J/(x)(I/(x))^n$ which implies $I^{n+1}\subseteq JI^n+(x)$.   
Hence $I^{n+1}\subseteq JI^n+(x)\cap I^{n+1}=JI^n+xI^n=JI^n$. So $r_J(I)\leq n$ which is a contradiction.
\end{proof}
We define $$\rho(I)=\min\{i\geq 1|\widetilde{I^n}=I^n \text{ for all } n\geq i\}.$$ 
As an interesting application of Lemma \ref{l1}, we see that Rossi's bound holds in dimension three for those $\m$-primary ideals $I$ for which $\rho(I)\leq r_J(I)-1$.  
\begin{proposition}\label{p1}
Let $(R,\m)$ be a Cohen-Macaulay local ring of dimension $d=3$ and $I$ an $\m$ primary ideal. For a minimal reduction $J$ of $I$, if $\rho(I)\leq r_J(I)-1$, then $r_J(I)\leq e_1(I)-e_0(I)+\lm(R/I)+1.$
\end{proposition}
\begin{proof}
Let $x\in I$ be a superficial element.  Suppose $\rho(I)\leq r_J(I)-1$.  Then $I^{r_J(I)-1}=\widetilde{I^{r_J(I)-1}}$ which implies  $r_{J/(x)}(I/(x))=r_J(I)$ by Lemma \ref{l1}.  Now, using the bound in \eqref{rossi-bound}, we get that  $r_J(I)=r_{J/(x)}(I/(x))\leq e_1(I)-e_0(I)+\ell(R/I)+1$. 
\end{proof}
 The following examples show that $r_{J/(x)}(I/(x))=r_J(I)$ may hold even if $\depth G(I)=0.$
\begin{example}\cite[Example 3.3]{rv}
Let $R=Q[[x,y]]$ and $I=(x^4,x^3y,xy^3,y^4)$. Then $x^2y^2\in I^2:I \subseteq \widetilde{I}$ but $x^2y^2\notin I$ which implies $\depth G(I)=0$. Note that $J=(x^4,y^4)$ is a minimal reduction of $I$ and $p=x^4+y^4$ is superficial for $I$ as $e_0(I)=16=e_0(I/(p))$ and $e_1(I)=6=e_1(I/(p))$. Further,
 $r_J(I)=2=r_{J/(p)}(I/(p)).$
\end{example}
\begin{example}\label{examplerv}\cite[Example 3.8]{cpr}
Let $R=Q[[x,y,z]]$ and $I=(x^2-y^2,y^2-z^2,xy,yz,xz).$ Then $\depth G(I)=0$ as $x^2\in(I^2:I)\subseteq\widetilde{I}$ but $x^2\notin I$. Using Macaulay 2, 
we find that  $J=(\frac{23}{6}x^2+\frac{1}{2}xy-\frac{5}{2}y^2+\frac{4}{3}xz+\frac{1}{3}yz-\frac{4}{3}z^2,~\frac{23}{63}x^2+\frac{10}{3}xy-\frac{2}{9}y^2+xz+9yz-\frac{1}{7}z^2,~6x^2+\frac{5}{6}xy-5y^2+\frac{5}{4}xz+\frac{7}{6}yz-x^2)$ is a minimal reduction of $I$ and 
$e_0(I)=8=e_0(I/(p))$, $e_1(I)=4=e_1(I/(p))$ and $e_2(I)=0=e_2(I/(p))$, where $p=\frac{23}{6}x^2+\frac{1}{2}xy-\frac{5}{2}y^2+\frac{4}{3}xz+\frac{1}{3}yz-\frac{4}{3}z^2$. This shows that $p$ is a superficial element for $I$. Further, $r_J(I)=2=r_{J/(p)}(I/(p))$.
\end{example}
\begin{lemma}\label{general-bound}
Let $(R,\m)$ be a  Noetherian local ring of dimension $d\geq 2$ and $I$ an $\m$-primary ideal with $\depth G(I^t)>0$ for some $t\geq 1.$ Let $x\in I$ be a superficial element for $I$ and $J\subseteq I$ be a minimal reduction of $I$.  Then
 \begin{equation*}
 r_J(I)\leq r_{J/(x)}(I/(x))+t-1. 
\end{equation*}
 Furthermore, if $r_J(I)\equiv k \mod t$, $1\leq k\leq t-1$, then 
  \begin{equation}\label{2nd-eq-p} 
  r_J(I)\leq  r_{J/(x)}(I/(x)) +k-1.\end{equation}
\end{lemma}
\begin{proof}
Since $\depth G(I^t)>0$, we have $\depth R>0$. We first consider the case when $r_J(I)\equiv k \mod t$ for $1\leq k\leq t-1$ and prove \eqref{2nd-eq-p}. Suppose $r_J(I)=mt+k$ for $m\geq 0$. We claim that $r_J(I)<r_{J/(x)}(I/(x))+ k$. Suppose $r_{J/(x)}(I/(x))\leq r_J(I)-k=mt$, then $r_{J/(x)}(I/(x))\leq  mt< r_J(I)$, but $\widetilde{I^{mt}}=I^{mt}$ as $\depth G(I^t)>0$.  Then by Lemma \ref{l1}, $r_J(I)=r_{J/(x)}(I/(x))$, a contradiction. Therefore, $$r_J(I)\leq r_{J/(x)}(I/(x))+ k-1.$$ 

Next, let $k=0$, i.e., $r_J(I)=mt$, $m\geq 1$. Then $r_J(I)<r_{J/(x)}(I/(x))+t$. Otherwise, $r_{J/(x)}(I/(x))\leq r_J(I)-t=(m-1)t<mt=r_J(I)$ and again $\widetilde{I^{(m-1)t}}=I^{(m-1)t}$ as $\depth G(I^t)>0$. Then by Lemma \ref{l1}, $r_J(I)=r_{J/(x)}(I/(x))$, a contradiction. Therefore, $$r_J(I)\leq r_{J/(x)}(I/(x))+t-1.$$
\end{proof}
We now generalize Rossi's result for $d=3$ case. Note that when $t=1$ in the result below, we obtain the $I$-adic case of \cite[Theorem 4.3]{rv} in dimension three.
\begin{theorem}\label{general-bound-three-dimension}
Let $(R,\m)$ be a Cohen-Macaulay local ring of dimension $d=3$ and $I$ an $\m$-primary ideal with $\depth G(I^t)>0$ for some $t\geq 1.$ Let $J\subseteq I$ be a minimal reduction of $I$.  Then
 \begin{equation*}\label{1st-eqn}r_J(I)\leq e_1(I)-e_0(I)+\lm(R/I)+t.\end{equation*}
 Furthermore, if $r_J(I)\equiv k \mod t$, $1\leq k\leq t-1$, then 
  \begin{equation*}\label{2nd-eq} r_J(I)\leq e_1(I)-e_0(I)+\lm(R/I)+k.\end{equation*}
\end{theorem}
\begin{proof}
Let $x\in I$ be a superficial element for $I$ and let $\ov R=R/(x)$. Then $\ov R$ is a two dimensional Cohen-Macaulay local ring.
By Lemma \ref{general-bound} and the bound in \eqref{rossi-bound}, we have  
\begin{eqnarray*}{}
r_J(I)& \leq &r_{J/(x)}(I/(x))+t-1\\
&\leq & e_1(I/(x))-e_0(I/(x))+\lm(R/(I+(x)))+t \\
& = & e_1(I)-e_0(I)+\lm(R/I)+t.
\end{eqnarray*}
When $r_J(I)\equiv k \mod t$, $1\leq k\leq t-1$, we have  $r_J(I) \leq r_{J/(x)}(I/(x))+k-1\leq  e_1(I)-e_0(I)+\lm(R/I)+k $ from \eqref{2nd-eq-p}.
\end{proof}
\begin{corollary}\label{cor-n-0}
Let $(R,\m)$ be a Cohen-Macaulay local ring of dimension $d=3$ and $I$ an $\m$-primary ideal. Let $J\subseteq I$ be a minimal reduction of $I$. Suppose $\depth G(I^2)>0$ and $r_J(I)$ is odd. Then
 $$r_J(I)\leq e_1(I)-e_0(I)+\lm(R/I)+1.$$
\end{corollary}
\begin{proof}
Since $\depth G(I^2)>0$ and $r_J(I)\equiv 1\mod 2$, the conclusion follows from Theorem \ref{general-bound-three-dimension}. 
\end{proof}
 For a graded ring $S$, let $H^i_{S_+}(S)$ denote the i-th local
cohomology module of $S$ with support in the graded ideal $S_+$ of elements of positive degree and set $a_i(S) = \max\{n~|~ H^i_{S_+}(S)_n \not= 0\}$ with the convention $a_i(S) = -\infty$ if $H^i_{S_+}(S) = 0.$ 
\begin{corollary}\label{cor-n-2}
Let $(R,\m)$ be a Cohen-Macaulay local ring of dimension $d=3$ and $I$ an $\m$-primary ideal. Let $J\subseteq I$ be a minimal reduction of $I$.  Then
 \begin{equation*}\label{1st-eqn-n}r_J(I)\leq\begin{cases} e_1(I)-e_0(I)+\lm(R/I)+1 & \text{ if } a_1(G(I))\leq 0\\
 e_1(I)-e_0(I)+\lm(R/I)+a_1(G(I))+1 & \text{ if } a_1(G(I))> 0
 \end{cases}\end{equation*}
\end{corollary}
\begin{proof}
Since  $\depth G(I^{\rho(I)})>0$, we can put $t=\rho(I)$ in Theorem \ref{general-bound-three-dimension} and $\rho(I)\leq \max\{a_1(G(I)+1,1\}$ by \cite[Theorem 4.3]{Put1}. 
\end{proof}

In \cite{itoh} Itoh proved that $e_2(\m)\geq e_1(\m)-e_0(\m)+1$. If $e_2(\m)=e_1(\m)-e_0(\m)+1\neq 0$, then $\tpe(R)\geq e_0(\m)-\mu(\m)+d-1$ where $\tpe(R)=\dim_k \Ext_R^d(k,R)$ and $\mu(\m)=\lm(\m/\m^2)$, see \cite{tm}. In \cite{tm}, the authors also proved that if $e_2(\m)=e_1(\m)-e_0(\m)+1\neq 0$ and $\tpe(R)= e_0(\m)-\mu(\m)+d-1$, then $G(\m)$ is Cohen-Macaulay. Therefore Rossi's bound as given by  \eqref{rossi-bound} holds for $r_J(\m)$ in this case. We consider the next boundary case, i.e., $\tpe(R)= e_0(\m)-\mu(\m)+d.$ In the corollary below, we obtain a linear bound in this case as well. 
\begin{corollary}\label{cor-n-1}
Let $(R,\m)$ be a Cohen-Macaulay local ring of dimension $d=3$ with $e_2(\m)=e_1(\m)-e_0(\m)+1\neq 0$ and $\tpe R=e_0(\m)-\mu(\m)+d$. Suppose $J\subseteq \m$ is a minimal reduction of $\m$. Then $r_J(\m)\leq e_1(\m)-e_0(\m)+\lm(R/\m)+3.$
\end{corollary}
\begin{proof}
If $\depth G(\m)\geq 1,$ then the conclusion follows from \cite[Theorem 4.3]{rv}.
Suppose $\depth G(\m)=0.$ By \cite[Theorem 4.2]{tm}, $\widetilde{\m^j}=\m^j$ for $j\geq 3$ which implies $\depth G(\m^3)>0.$ Then by Theorem \ref{general-bound-three-dimension}, $r_J(\m)\leq e_1(\m)-e_0(\m)+\lm(R/\m)+3$.
\end{proof}

We now consider Example \ref{example-n} with $m=0$ and $d=3$ to demonstrate that the bound in Theorem \ref{general-bound-three-dimension} is better than the one given by Vasconcelos in \eqref{theorem-intro-old1}. 
\begin{example}\label{exfromtm}
Let $R=k[[x,y,z,u,v,w,t]]/(t^2,tu,tv,tw,uv,uw,vw,u^3-xt,v^3-yt,w^3-zt)$. Then $R$ is a Cohen-Macaulay local ring of dimension $3$ and $\depth G(\m)=0.$ We have $e_0(\m)=8$, $e_1(\m)=11$, $e_2(\m)=4$, $e_3(\m)=0$ and $\type(R)=e_0(\m)-\mu(\m)+3=4$, see \cite[Example 5.2(1)]{tm}. By \cite[Theorem 4.2]{tm}, we have $\m^2\neq \widetilde{\m^2}$ and $\m^j= \widetilde{\m^j}$ for $j\geq 3$. Therefore $\depth G(\m^3)\geq 1$. 
Now $J=(x,y,z)$ is a minimal reduction of $\m$ and $r_J(\m)=3\leq e_1(\m)-e_0(\m)+\lm(R/\m)+3=7.$ 
Note that the bound $\frac{de_0(\m)}{o(\m)}-2d+1=3.8-6+1=19$ given by Vasconcelos  in \cite{v}  is larger than our bound.  
\end{example}


In the next proposition, we summarize the cases when Rossi's bound  holds in dimension three. Some of these results are already known. 
Let $v_n(\mathcal I)=\lm(I_{n+1}/JI_n)$ for any $I$ admissible filtration $\mathcal I$ and $\mathcal F=\{\widetilde{I^n}\}_{n\in\mathbb{Z}}$ denote the Ratliff-Rush filtration. By the proof of Rossi's result \cite[Theorem 1.3]{r} in a $d$ dimensional Cohen-Macaulay local ring,  we have 
\begin{equation}\label{main-equation-for-bound}
    r_J(I)\leq \sum_{n\geq 0}v_n(\mathcal F)-e_0(I)+\lm(R/I)+1
    \end{equation}
    The idea in the next few results is to approximate the term $\sum_{n\geq 0}v_n(\mathcal F).$
\begin{proposition}\label{prop-few-cond-for-rossi-bound}
Let $(R,\m)$ be a three dimensional Cohen-Macaulay local ring, $I$ an $\m$ primary ideal and $J$ a minimal reduction of $I$. Then $r_J(I)\leq e_1(I)-e_0(I)+\lm(R/I)+1$ if one of the following conditions hold:
 \begin{enumerate}[label=(\roman*)]
     \item \label{part-1} $\depth G(\mathcal F)\geq 2.$ 
     \item \label{part-3} $e_2(I)=e_3(I)=0.$ 
     \item $e_2(I)=0$ and $I$ is asymptotically normal .
     \item $e_2(I)=0$ and $G(I)$ is generalized Cohen-Macaulay. 
     \item $\rho(I)\leq r_J(I)-1.$
     \item $a_1(G(I))\leq 0.$
 \end{enumerate}
\end{proposition}
\begin{proof}
\begin{enumerate}[label=(\roman*)]
    \item 
As $\depth G(\mathcal F)\geq 2$,  $e_1(I)=e_1(\mathcal{F})= \displaystyle{\sum_{n\geq 0}v_n(\mathcal F)}$ by \cite[Proposition 4.6]{hm}. Substituting this into (\ref{main-equation-for-bound}), we get $r_J(I)\leq e_1(I)-e_0(I)+\lm(R/I)+1$.

\item 
If $e_2(I)=e_3(I)=0$, then $G(\mathcal F)$ is Cohen-Macaulay by \cite[Theorem 6.2]{Put2}) and hence the conclusion follows from part \ref{part-1}. 
\item By \cite[Theorem 4.1]{cpr}, $e_3(I)\geq 0$ for an asymptotically normal ideal $I$ and by \cite[Proposition 6.4]{Put2}, $e_2(I)=0$ implies $e_3(I)\leq 0$. This gives $e_3(I)=0$. Now the conclusion follows from part \ref{part-3}. 
\item 
 Suppose $e_2(I)=0$. Then $e_3(I)=0$ if and only if $G(I)$  is generalized Cohen-Macaulay by \cite[Proposition 6.4]{Put2}. Now, the conclusion follows from part \ref{part-3}.
\item It follows from Proposition \ref{p1}. 
\item It follows from Corollary \ref{cor-n-2}.
\end{enumerate}
\end{proof}
\begin{remark}
\begin{enumerate}
    \item Note that in Example \ref{exfromtm},  $G(\widetilde{\m^n})=\mathop\oplus\limits_{n\geq 0} \widetilde{\m^n}/\widetilde{\m^{n+1}} $ is Cohen-Macaulay by \cite[Theorem 4.2]{tm}. Hence by Proposition \ref{prop-few-cond-for-rossi-bound}\ref{part-1}, we have $3=r_J(\m)\leq e_1(\m)-e_0(\m)+\lm(R/\m)+1=5$. 
    \item In Example \ref{examplerv}, we have  $e_2(I)=e_3(I)=0$. Hence by Proposition \ref{prop-few-cond-for-rossi-bound}\ref{part-3},  $2=r_J(I)\leq e_1(I)-e_0(I)+\lm(R/I)+1=2$. 
\end{enumerate}
\end{remark}
Next we give an upper bound for the reduction number of an ideal in Buchsbaum local ring with dimension at most two.  
\begin{theorem}\label{1dim-Buchsbaum}
Let $(R,\m)$ be a one dimensional Buchsbaum local ring and $I$ an $\m$-primary ideal. Let $J$ be a  minimal reduction of $I$, then
$$r_J(I)\leq e_1(I)-e_1(J)-e_0(I)+\lm(R/I)+2.$$
\end{theorem}
\begin{proof}
 Let $S=R/H^0_{\m}(R)$. Let $r=r_{JS}(IS)$. Then $I^{r+1}-JI^r\subseteq H^0_{\m}(R)$, which implies that $I^{r+2}-JI^{r+1}\subseteq IH^0_{\m}(R)=0$. Hence 
 \begin{equation}\label{reduction}
     r_J(I)\leq r_{JS}(IS)+1.
 \end{equation}
  Note that $S$ is a 1 dimensional Cohen-Macaulay local ring. Therefore, by \eqref{rossi-bound}, we have
 \begin{eqnarray}
  r_J(I)&\leq& r_{JS}(IS)+1\nonumber \\ &\leq& e_1(IS)-e_0(IS)+\lm(S/IS)+2 \nonumber \\
 &\leq & e_1(I)-e_1(J)-e_0(I)+\lm(R/I)+2 \mbox{ (by \cite[Lemma 2.3, Proposition 2.3]{rv})}\nonumber \label{buchsbaum}
 \end{eqnarray}
 \end{proof}
\begin{theorem}\label{2dim-Buchsbaum}
Let $(R,\m)$ be a two dimensional Buchsbaum local ring and $I$ an $\m$-primary ideal. Let $J$ be a minimal reduction of $I$ and $\depth G(I^t)>0$ for some $t\geq 1$, then
$$r_J(I)\leq e_1(I)-e_1(J)-e_0(I)+\lm(R/I)+t+1$$
\end{theorem}
\begin{proof}
Note that $\depth R>0$ as $\depth G(I^t)>0$. 
Let $x\in I$ be a superficial element for $I$. Then by Lemma \ref{general-bound}, we have $r_J(I)\leq r_{J/(x)}(I/(x))+t-1$. Since $R/(x)$ is a one dimensional Buchsbaum local ring, by Theorem \ref{1dim-Buchsbaum} we have 
\begin{eqnarray*}
r_{J/(x)}(I/(x))&\leq & e_1(I/(x))-e_1(J/(x))-e_0(I/(x))+\lm(R/I)+2\\
& = & e_1(I)-e_1(J)-e_0(I)+\lm(R/I)+2
\end{eqnarray*}
Therefore
$r_J(I)\leq e_1(I)-e_1(J)-e_0(I)+\lm(R/I)+t+1.$
\end{proof}
\section{Bound for $r_J(I)$ in dimension three}\label{section-4}
In this section we give a different upper bound for reduction number of $I$ in  a Cohen-Macaulay local ring of dimension $d\geq 3$  when $\depth G(I)\geq d-3$. Our bound involves $e_2(I)$ and $e_3(I)$. For an $I$ admissible filtration $\mathcal I=\{I_n\}_{n\in\mathbb{Z}}$, let us denote by $\mathcal R(\mathcal I)=\mathop\bigoplus\limits_{n\geq 0}I_n$ the Rees algebra of $\mathcal I$.  The second Hilbert function of $\mathcal I$ is defined as $H^2_{\mathcal I}(n)=\mathop\sum\limits_{i=0}^n H_{\mathcal I}(i)$ and the second Hilbert polynomial, denoted by $P^2_{\mathcal I}(n)$ is the polynomial which coincides with $H^2_{\mathcal I}(n)$ for large values of $n$. It is well known that the Hilbert series of $\mathcal I$, defined as $H(\mathcal I,z)=\sum_{n\geq 0} \lm(I_n/I_{n+1})z^n$, is rational, i.e., there exists a unique rational polynomial $h_{\mathcal I}(z)\in\mathbb{Q}[z]$ with $h_{\mathcal I}(1)\neq 0$ such that $$H(\mathcal I,z)=\frac{h_{\mathcal I}(z)}{(1-z)^d}.$$ For every $i\geq 0$, we define $e_i(\mathcal I)=\frac{h_{\mathcal I}^{(i)}(1)}{i!}$, where $h_{\mathcal I}^{(i)}(1)$ denotes the $i$-th formal derivative of the polynomial $h_{\mathcal I}(z)$ at $z=1$. The integers $e_i(\mathcal I)$ are called the Hilbert coefficients of  $\mathcal I$ and for $0\leq i\leq d$, these are same as defined earlier in the Introduction, see \cite{gr} for more details.

Let us recall the modified Koszul complex in dimension two defined in \cite{marley} as follows:
$$C_.({\mathcal I},n): ~~~~~ 0\longrightarrow R/{{ I}_{n-2}} \overset{\begin{psmallmatrix} -y \\ ~x \end{psmallmatrix}}\longrightarrow ( R/{ I}_{n-1})^2 \overset{( x, y)}\longrightarrow R/{I}_n\longrightarrow 0,$$ where $(x,y)$ is a minimal reduction of $I$. Let $H_i(C_.(\mathcal I,n))$ denote the $i$-th homology module of the complex $C_.(\mathcal I,n)$. The relation between the homology of this complex and Hilbert coefficients is used in the proof of the next theorem. For a numerical function $f:\mathbb  Z\longrightarrow  \mathbb Z$, we put $\triangle f(n) =f(n+1)-f(n)$ and recursively we can define $\triangle^i f(n):=\triangle(\triangle^{i-1} f(n))$ for all $i\geq 1$.
\begin{theorem}\label{proposition-new-bound-1}
Let $(R,\m)$ be a  Cohen-Macaulay local ring of dimension $d\geq 3$ and $I$  an $\m$- primary ideal  with $\depth G(I)\geq d-3$. Then \begin{equation}\label{new-bound1}
    r_J(I)\leq e_1(I)-e_0(I)+\lm(R/I)+1+(e_2(I)-1)e_2(I)-e_3(I).
\end{equation}
\end{theorem}
\begin{proof}
Suppose $d=3$. Let $\mathcal F=\{\widetilde{I^n}\}$, $x\in I$ be a superficial element for $I$ and $J=(x,y,z)$ a minimal reduction of $I$. Then $x$ is also superficial for the filtration $\mathcal F$. Let $\ov R=R/(x)$ and $\ov{\mathcal F}=\{\ov{\mathcal F}_n=\frac{\widetilde{I^n}+(x)}{(x)}\}$. Since $\depth G(\mathcal F)\geq 1$, we have $v_n(\mathcal F)=v_n(\ov{\mathcal F})$.
By the proof of \cite[Proposition 2.9]{h}, we have 
\begin{eqnarray}
e_1(\ov{\mathcal F})&=&\sum_{n\geq 1}\triangle^2(P_{\ov{\mathcal F}}(n)-H_{\ov{\mathcal F}}(n))  \nonumber \\
&=& \sum_{n\geq 1}\Big(e_0(\ov{\mathcal F})-\sum_{i= 0}^2 (-1)^i \lm(H_i(C_.(\ov{\mathcal F},n)))\Big)   \mbox{ (by \cite[Proposition 3.2]{marley})}\nonumber \\
&=& \sum_{n\geq 0}v_n(\ov{\mathcal F})-\sum_{n\geq 1}\lm(H_2(C_.(\ov{\mathcal F},n))) \mbox{ (by the proof of \cite[Theorem 3.6]{marley})}.\label{eq}
\end{eqnarray}
Since $x$ is a superficial element for $\mathcal F$,  $e_1(I)=e_1(\mathcal F)=e_1(\ov{\mathcal F})$.
Therefore, by using \eqref{main-equation-for-bound} and (\ref{eq}), we get 
\begin{equation}\label{general-stat2}
    r_J(I)\leq e_1(I)+\sum_{n\geq 1}\lm(H_2(C_.(\ov{\mathcal F},n)))-e_0(I)+\lm(R/I)+1.
    \end{equation}
From the modified Koszul complex $C_.(\ov{\mathcal F},n)$, we have  $H_2(C.(\ov{\mathcal F},n))=\frac{\ov{\mathcal F}_{n-1}:(\ov y,\ov z)}{\ov{\mathcal F}_{n-2}}$.
Since $\ov{\mathcal F}_{n-1}:(\ov y,\ov z)\subseteq \widetilde{\ov{\mathcal F}}_{n-2}$,  $$\lm(H_2(C.(\ov{\mathcal F},n))\leq \lm\left (\frac{\widetilde{\ov{\mathcal F}}_{n-2}}{\ov{\mathcal F}_{n-2}}\right ).$$
Therefore,  for large $m$ we have
\begin{eqnarray*}
0\leq \sum_{n= 0}^m \lm(H_2(C.(\ov{\mathcal F},n)) &\leq & \sum_{n= 0}^m \lm \left(\frac{\widetilde{\ov{\mathcal F}}_{n-2}}{\ov{\mathcal F}_{n-2}}\right )\\
&=&   \sum_{n= 0}^m \lm(\ov R/\ov{\mathcal F}_{n-2})-\sum_{n=0}^m\lm(\ov R/\widetilde{\ov{\mathcal F}}_{n-2})\\
&=&  e_3(\widetilde{\ov{\mathcal F}}) -e_3(\ov{\mathcal F})\\
&=&   e_3(\widetilde{\ov{\mathcal F}}) -e_3(\mathcal F) ~~\mbox{ (by \cite[Proposition 1.5]{gr})}\\
&=&  e_3(\widetilde{\ov{\mathcal F}}) -e_3(I)
\end{eqnarray*}
This gives
\begin{equation}\label{e3}
0\leq \sum_{n\geq 0}\lm(H_2(C_.(\ov{\mathcal F},n)))\leq e_3(\widetilde{\ov {\mathcal F}})-e_3(I).
\end{equation}
From (\ref{general-stat2}) and (\ref{e3}), we get
\begin{equation}\label{3}
r_J(I)\leq e_1(I)-e_0(I)+\lm(R/I)+1+e_3(\widetilde{\ov{\mathcal F}})-e_3(I).
\end{equation}
By the difference formula in \cite[Proposition 4.4]{blanc},  we have
for all $n\geq -1$,
\begin{equation}\label{diff}
P_{\widetilde{\ov{\mathcal F}}}(n)-H_{\widetilde{{\ov{\mathcal F}}}} (n)=\lm((H^2_{\mathcal R_+}(\mathcal R(\widetilde{{\ov {\mathcal F}}})))_{n+1}). 
\end{equation}
Now taking sum for large $m$ on both sides of the above equation, we get
\begin{eqnarray*}
\sum_{n=0}^m \lm((H^2_{\mathcal R_+}(\mathcal R(\widetilde{{\ov {\mathcal F}}})))_{n+1})&=& \sum_{n=0}^m P_{\widetilde{\ov{\mathcal F}}}(n)-\sum_{n=0}^m H_{\widetilde{{\ov{\mathcal F}}}} (n)\\
&=& \sum_{n=0}^m P_{\widetilde{\ov{\mathcal F}}}(n)- H^2_{\widetilde{{\ov{\mathcal F}}}} (m)
\\
&=& e_0(\widetilde{{\ov{\mathcal F}}}){m+3\choose 3}-e_1(\widetilde{{\ov{\mathcal F}}}){m+2\choose 2}+e_2(\widetilde{{\ov{\mathcal F}}}){m+1\choose 1}- P^2_{\widetilde{{\ov{\mathcal F}}}} (m)\\
&=& e_3(\widetilde{\ov{\mathcal F}}).
\end{eqnarray*}
As $\ov R$ is a $2$-dimensional Cohen-Macaulay local ring,  we have $\lm((H^2_{\mathcal R_+}(\mathcal R(\widetilde{{\ov {\mathcal F}}})))_n)\leq \lm((H^2_{\mathcal R_+}(\mathcal R(\widetilde{{\ov {\mathcal F}}})))_{n-1})$ for all $n\in \mathbb Z$ by \cite[Lemma 4.7]{blanc}. Now in equation $(\ref{diff})$, we substitute $n=-1$ to get $$\lm((H^2_{\mathcal R_+}(\mathcal R(\widetilde{{\ov {\mathcal F}}})))_0)=e_2(\widetilde{{\ov {\mathcal F}}})=e_2(\ov{\mathcal F})=e_2(\mathcal F)=e_2(I)$$
Therefore, 
\begin{equation}\label{e3tilde}
e_3(\widetilde{\ov {\mathcal F}})=\sum_{n=0}^m \lm((H^2_{\mathcal R_+}(\mathcal R(\widetilde{{\ov {\mathcal F}}})))_{n+1})\leq\sum_{n=0}^ {a_2(\mathcal R(\widetilde{\ov {\mathcal F}}))-1} \lm((H^2_{\mathcal R_+}(\mathcal R(\widetilde{{\ov {\mathcal F}}})))_0)= a_2(\mathcal R(\widetilde{\ov {\mathcal F}}))e_2(I)
\end{equation}
where $a_2(\mathcal R(\widetilde{\ov {\mathcal F}}))\leq a_2(G(\widetilde{\ov {\mathcal F}}))=s$ (say). 
Now using (\ref{3}) and (\ref{e3tilde}), we have 
\begin{equation}\label{general-state3}
r_J(I)\leq e_1(I)-e_0(I)+\lm(R/I)+1+se_2(I)-e_3(I).
\end{equation}
By \cite[Corollary 5.7(2)]{marley}, we have $s=a_2(G(\widetilde{\ov {\mathcal F}}))=r(\widetilde{\ov {\mathcal F}})-2$ and 
by Theorem \ref{lemma-bound-on-rtilde}, $r_J(\widetilde{\ov{\mathcal{F}}})-1\leq  e_2(\ov{\mathcal{F}})$. This gives  
\begin{equation}\label{ainvariant}
s=r_J(\widetilde{\ov{\mathcal{F}}})-2\leq e_2(\ov{\mathcal{F}})-1= e_2(I)-1. 
\end{equation}
Now by (\ref{general-state3}) and (\ref{ainvariant}),  we get the conclusion. 

Suppose $d\geq 4.$ Let $x\in I$ be a superficial element for $I$. Then $e_i(I/(x))=e_i(I)$ for $0\leq i\leq 3$. Also, $\depth G(I)\geq 1$ implies $I^n=\widetilde{I^n}$ for $n\geq 1.$ This gives  $r_{J/(x)}(I/(x))=r_J(I)$ by Lemma \ref{l1}. This completes the proof. 
\end{proof}
\begin{example}
\begin{enumerate}
    \item 
We refer to Example \ref{exfromtm} to note that our bound $r_J(\m)\leq e_1(\m)-e_0(\m)+\lm(R/\m)+1+(e_2(\m)-1)e_2(\m)-e_3(\m)=17$ is better than Vasconcelos' bound $\frac{d e_0(\m)}{o(\m)}-2d+1=19.$ 
\item Example \ref{example-n} provides a number of three dimensional Cohen-Macaulay local rings with $e_1(\m)-e_0(\m)+\lm(R/\m)+1+(e_2(\m)-1)e_2(\m)-e_3(\m)=17$ and
$\frac{3 e_0(\m)}{o(\m)}-2.3+1=3m+19$, where $m\geq 0$.

\item  Let $R=Q[|x,y,z|]$ and $I = (x^4,y^4,z^4,x^3y,xy^3,y^3z,yz^3)$. Note that $\depth G(I)=0$ as $x^2y^2z^2\in I^2:I\subseteq \widetilde{I}$ but $x^2y^2z^2\notin I$. By \cite[Example 3.7]{mn},  $e_0(I)=64$, $e_1(I)=48$, $e_2(I)=4$ and $e_3(I)=0$. Using Macaulay 2, $J=(\frac{3}{4}x^4+x^3y+\frac{5}{6}xy^3+\frac{1}{2}y^4+5y^3z+3yz^3+\frac{1}{5}z^4,~\frac{4}{3}x^4+\frac{5}{3}x^3y+\frac{1}{2}xy^3+\frac{3}{2}y^4+\frac{4}{5}y^3z+\frac{5}{3}yz^3+\frac{10}{7}z^4,~\frac{5}{3}x^4+\frac{4}{5}x^3y+\frac{7}{2}xy^3+y^4+\frac{3}{2}y^3z+\frac{8}{9}yz^3+\frac{3}{4}z^4)$ is a  minimal reduction of $I$ and $3=r_J(I)\leq e_1(I)-e_0(I)+\lm(R/I)+1+(e_2(I)-1)e_2(I)-e_3(I)=32$ whereas $\frac{d e_0(I)}{o(I)}-2d+1=43.$
\end{enumerate}
\end{example}
Next we show that in dimension three, for certain values of $e_2(I)$, we get  linear upper bound on $r(I)$ in terms of Hilbert coefficients. We write $v_n$ for $v_n(\widetilde{\ov{\mathcal F}})$. 
\begin{corollary}\label{corr-main-result-1}
Let $(R,\m)$ be a three dimensional Cohen-Macaulay local ring and $I$ an $\m$-primary ideal. Then the following statements hold.
\begin{enumerate}
    \item If $e_2(I)=0$ or $1$ then 
$r(I)\leq e_1(I)-e_0(I)+\lm(R/I)+1-e_3(I)$. 
\item If $e_2(I)=1$ and $I$ is asymptotically normal then 
$r(I)\leq e_1(I)-e_0(I)+\lm(R/I)+1$
\item If $e_2(I)=2$ then 
$r(I)\leq e_1(I)-e_0(I)+\lm(R/I)+2-e_3(I)$.
\end{enumerate}
\end{corollary}
\begin{proof}
\begin{enumerate}
\item If $e_2(I)=0$ or $1$ then $r(I)\leq e_1(I)-e_0(I)+\lm(R/I)+1-e_3(I)$ by Theorem \ref{proposition-new-bound-1}. 

\item If $e_2(I)=1$ then  $s\leq 0$ by (\ref{ainvariant}) and hence $e_3({\widetilde{\ov{\mathcal F}}})\leq 0$ using (\ref{e3tilde}). Then by equation (\ref{e3}), $e_3(I)\leq 0$. Since $I$ is asymptotically normal, $e_3(I)\geq 0$ by \cite[Theorem 4.1]{cpr} which implies $e_3(I)=0$. Hence we have $r(I)\leq e_1(I)-e_0(I)+\lm(R/I)+1$.

\item As $\depth G(\widetilde{\ov{\mathcal F}})\geq 1$, we have  $2=e_2(I)=e_2(\widetilde{\ov{\mathcal F}})=\sum_{n\geq 1}nv_n$ (by \cite[Theorem 2.5]{rv}),
which implies either $v_1=0,v_2=1,v_n=0$ for all $n\geq 3$ or $v_1=2,v_n=0$ for all $n\geq 2$.  Hence $e_3(\widetilde{\ov{\mathcal F}})=\sum_{n\geq 2}{n\choose 2}v_n=v_2\leq 1$. Using  (\ref{3}), we get $r(I)\leq e_1(I)-e_0(I)+\lm(R/I)+2-e_3(I)$.
\end{enumerate}
\end{proof}

{\bf Acknowledgement :} We would like to express our sincere gratitude to anonymous referee for meticulous reading and suggesting several editorial improvements. We also thank Prof. M. E. Rossi for her suggestions.

\end{document}